\documentclass{amsart}

\usepackage{amssymb,amsmath,amsthm,latexsym,stmaryrd}
\usepackage{color}
\usepackage{multirow}
\usepackage{hhline}
\usepackage[inline]{enumitem}

\newtheorem{thm}{Theorem}[section]
\newtheorem{theorem}{Theorem}[section]
\newtheorem*{theorem A}{Theorem A}
\newtheorem*{theorem B}{N\"olker's Theorem}

\theoremstyle{remark}

\theoremstyle{remark}

\theoremstyle{definition}

\numberwithin{equation}{section}
\def\({\left ( }
    \def\){\right )}
\def\<{\left < }
\def\>{\right >}

\newcommand{\ie}{i.e. }
\newcommand{\s}{\mathfrak{S}}

\newcommand{\n}{\nabla}

\newcommand{\W}{\mathcal{W}}
\newcommand{\MM}{\mathcal{M}}
\newcommand{\ea}{\varepsilon_\alpha}

\newcommand{\R}{\mathbb R}

\newcommand{\ta}{\theta}

\newcommand{\gm}{\gamma}
\newcommand{\al}{\alpha}
\newcommand{\bt}{\beta}

\begin{document}

\vspace{2cm}

\title[Lie groups of dimension 4 and almost hypercomplex manifolds \dots]
{Lie groups of dimension 4 \\and almost hypercomplex manifolds \\with Hermitian-Norden metrics}

\author{Hristo Manev}
\address{Medical University -- Plovdiv, Faculty of Public Health,
Department of Medical Informatics, Biostatistics and e-Learning,   15-A Vasil Aprilov
Blvd.,   Plovdiv 4002,   Bulgaria;}
\email{hristo.manev@mu-plovdiv.bg}

\subjclass[2010]{Primary: 53C15, 53C50; Secondary: 22E60, 22E15, 53C55}



\keywords{Almost hypercomplex structure, Hermitian metric, Norden metric, Lie group, Lie algebra}

\begin{abstract}
Object of investigation are almost hypercomplex manifolds with Hermitian-Norden metrics of the lowest dimension. The considered manifolds are constructed on 4-dimensional Lie groups. It is established a relation between the classes of a classification of 4-dimensional indecomposable real Lie algebras and the classification of the manifolds under study. The basic geometrical characteristics of the constructed manifolds are studied in the frame of the mentioned classification of the Lie algebras.\thanks{The author was partially supported by Project MU21-FMI-008 of the Scientific Research Fund, University of Plovdiv, Bulgaria and National Scientific Program ''Young Researchers and Post-Doctorants'', Bulgaria}
\end{abstract}
\maketitle
\section{Introduction}
\label{intro}
A triad of anticommuting almost complex structures
such that each of them is a composition of the other two structures is called an almost hypercomplex structure $H$ on a $4n$-dimensional smooth manifold $\MM$. The structure $H$ could be equipped with a metric structure of Hermitian-Norden type, generated by a pseudo-Riemannian metric $g$ of neutral signature (\cite{GriMan24,GriManDim12}). In this case, in each tangent fibre, one of the almost complex structures of $H$ acts as an isometry and the other two act as anti-isometries with respect to $g$.
The metric $g$ is Hermitian with respect to one of almost complex structures of $H$ and $g$ is a Norden metric regarding the other two. Then, we have three associated (0,2)-tensors to the metric $g$ -- a K\"ahler form and two Norden metrics.

The manifold $\MM$, equipped with the considered structures, is called an almost hypercomplex manifold with Her\-mit\-ian-Norden metrics. The same manifolds are investigated in \cite{GriMan24,GriManDim12} under the name almost hypercomplex pseudo-Hermitian manifolds and in \cite{Man28,ManGri32} as almost hypercomplex manifolds with Hermitian and anti-Hermitian metrics.

Almost hypercomplex manifolds with Her\-mit\-ian-Norden metrics can be con\-struct\-ed on Lie groups.
In this work we use classification of four-dimensional indecomposable Lie algebras, known from \cite{GhaTho}.
The goal of this paper is to find a relation between the classes in this classification
and the corresponding manifolds to the classifications given in \cite{GrHe} and \cite{GaBo}, which are derived by the tensor structures and metrics of the respective manifolds.
Moreover, the present work gives the basic geometrical characteristics of the considered manifolds in each case.

The author's intention with this article is to complete the considered problem for all classes of the mentioned classification and thus to generalize the results from \cite{HM12} and \cite{HM13}.

Smooth manifolds with similar structures on Lie groups are studied in \cite{Barb,ManTav2,Ovando,ZamNak}.

\section{Almost hypercomplex manifolds with Hermit\-ian-Norden metrics}\label{sect-prel}

The subject of our study are \emph{almost hypercomplex
manifolds with Hermit\-ian-Norden metrics} (\cite{GriManDim12}). A differentiable manifold $\MM$ of this type has dimension $4n$ and it is denoted by $(\MM,H,G)$, where $(H,G)$ is an \emph{almost hy\-per\-com\-plex
structure with Hermit\-ian-Norden metrics}. More precisely, the almost hypercomplex structure $H=(J_1,J_2,J_3)$ has the following properties:
\begin{equation*}\label{J123} %
J_\al=J_\bt\circ J_\gm=-J_\gm\circ J_\bt, \qquad
J_\al^2=-I,
\end{equation*} %
for all cyclic permutations $(\alpha , \beta , \gamma )$ of $(1, 2, 3)$ and the identity $I$.
The quadruplet $G=(g,g_1,g_2,g_3)$ consists of a a neutral metric $g$, associated 2-form $g_1$ and associated neutral metrics $g_2$ and $g_3$ on $(\MM,H)$ having the properties %
\begin{equation}\label{gJJ} %
g(\cdot,\cdot)=\ea g(J_\al \cdot,J_\al \cdot),
\end{equation}
\begin{equation}\label{gJ} %
g_\al(\cdot,\cdot)=g(J_\al \cdot,\cdot)=-\ea g(\cdot,J_\al \cdot).
\end{equation} %
where
\begin{equation*}\label{epsiloni}
\ea=
\begin{cases}
\begin{array}{ll}
1, \quad & \al=1;\\
-1, \quad & \al=2;3.
\end{array}
\end{cases}
\end{equation*}

Here and further, $\alpha$ will run over the range
$\{1,2,3\}$ unless otherwise is stated.

Let us remark that the considered type of manifolds is the only possible way to involve Norden-type metrics on almost hypercomplex manifolds.

The following three tensors of type $(0,3)$ are the fundamental tensors of the almost hypercomplex
manifold with Hermit\-ian-Norden metrics (\cite{GriManDim12})
\begin{equation}\label{F'-al}
F_\al (x,y,z)=g\bigl( \left( \n_x J_\al
\right)y,z\bigr)=\bigl(\n_x g_\al\bigr) \left( y,z \right),
\end{equation}
where $\n$ is the Levi-Civita connection of $g$.
These tensors have the properties
\begin{equation}\label{FaJ-prop}
  F_{\al}(x,y,z)=-\ea F_{\al}(x,z,y)=-\ea F_{\al}(x,J_{\al}y,J_{\al}z)
\end{equation}
and they are related to each other as follows
\begin{equation*}\label{F1F2F3}
\begin{array}{l}
    F_1(x,y,z)=F_2(x,J_3y,z)+F_3(x,y,J_2z),\\[6pt]
    F_2(x,y,z)=F_3(x,J_1y,z)+F_1(x,y,J_3z),\\[6pt]
    F_3(x,y,z)=F_1(x,J_2y,z)-F_2(x,y,J_1z).
\end{array}
\end{equation*}

The corresponding 1-forms $\ta_\al$ of $F_\al$, known as Lee forms, are determined by
\begin{equation}\label{theta-al}
\ta_\al(\cdot)=g^{ij}F_\al(e_i,e_j,\cdot),
\end{equation}%
where $\{e_1,e_2,\dots, e_{4n}\}$ is an arbitrary basis of $T_p\MM$,
$p\in \MM$ and $g^{ij}$ are the corresponding components of the inverse matrix of $g$.

According to \eqref{gJJ}, $(\MM,J_1,g)$ is an almost Hermitian manifold whereas the manifolds $(\MM,J_2,g)$ and $(\MM,J_3,g)$ are almost complex manifolds with Norden metric. These two types of manifolds are classified in \cite{GrHe} and \cite{GaBo}, respectively.
In the case of the lowest dimension 4, the four basic classes of almost Hermitian manifolds with respect to
$J_1$ are restricted to two:
\begin{equation}\label{cl-H-dim4}
\begin{split}
&\W_2(J_1):\; \mathop{\s}_{x,y,z}\bigl\{F_1(x,y,z)\bigr\}=0; \\
&\W_4(J_1):\; F_1(x,y,z)=\dfrac{1}{2}
                \left\{g(x,y)\ta_1(z)-g(x,J_1y)\ta_1(J_1z)\right. \\
&\phantom{\W_4(J_1):\; F_1(x,y,z)=\quad\,}
                \left.-g(x,z)\ta_1(y)+g(x,J_1z)\ta_1(J_1y)
                \right\},
\end{split}
\end{equation}
where $\s $ is the cyclic sum by three arguments.
In the 4-dimensional case, the basic classes of almost Norden manifolds ($\al=2$ or $3$)
are determined as follows:
\begin{equation}\label{cl-N-dim4}
\begin{split}
&\W_1(J_\al):\; F_\al(x,y,z)=\dfrac{1}{4}\bigl\{
g(x,y)\ta_\al(z)+g(x,J_\al y)\ta_\al(J_\al z)\bigr.\\
&\phantom{\W_1(J_\al):\; F_\al(x,y,z)=\quad\,\,} %
\bigl.+g(x,z)\ta_\al(y)
    +g(x,J_\al z)\ta_\al(J_\al y)\bigr\};\\
&\W_2(J_\al):\; \mathop{\s}_{x,y,z}
\bigl\{F_\al(x,y,J_\al z)\bigr\}=0,\qquad \ta_\al=0;\\
&\W_3(J_\al):\; \mathop{\s}_{x,y,z} \bigl\{F_\al(x,y,z)\bigr\}=0.
\end{split}
\end{equation}

The curvature (1,3)-tensor of $\nabla$ is defined as usual by $R=\left[\n,\n\right]-\n_{[\ ,\ ]}$. The corresponding curvature (0,4)-tensor with respect to $g$ is denoted by the same letter, \ie
\begin{equation}\label{R}
R(x,y,z,w)=g(R(x,y)z,w),
\end{equation}
and it has the following well-known properties:
\begin{equation}\label{R-prop}
\begin{array}{c}
    R(x,y,z,w)=-R(y,x,z,w)=-R(x,y,w,z), \\[3pt]
R(x,y,z,w)+R(y,z,x,w)+R(z,x,y,w)=0.
\end{array}
\end{equation}

The Ricci tensor $\rho$ and the scalar curvature $\tau$ for $R$ as well as
their associated quantities $\rho^*$, $\tau_\al^*$ and $\tau_\al^{**}$ are defined by:
\begin{equation*}
\begin{array}{c}
    \rho(y,z)=g^{ij}R(e_i,y,z,e_j),\qquad \rho_\al^*(y,z)=g^{ij}R(e_i,y,z,J_\al e_j),\\[6pt]
    \tau=g^{ij}\rho(e_i,e_j), \qquad \tau_\al^*=g^{ij}\rho_\al^*(e_i,e_j),\qquad \tau_\al^{**}=g^{ij}\rho_\al^*(e_i,J_\al e_j).
\end{array}
\end{equation*}
The following properties for $\rho$ and $\rho_\al^*$ are valid:
\begin{equation}\label{rho-prop}
\begin{array}{c}
\rho_{jk}=\rho_{kj}$, \quad $(\rho_\al^*)_{jk}=-\ea (\rho_\al^*)_{kj},
\end{array}
\end{equation}
where $\rho_{jk}=\rho(e_j,e_k)$ and  $(\rho_\al^*)_{jk}=\rho_\al^*(e_j,e_k)$ are the basic components of $\rho$ and $\rho_\al^*$, respectively.

Let $\mu$ be a non-degenerate 2-plane with a basis $\{x,y\}$ in
$T_p\MM$, $p \in \MM$. The sectional curvature of $\mu$ with respect to $g$ and $R$ is defined by
\[
k(\mu;p)=\frac{R(x,y,y,x)}{g(x,x)g(y,y)-g(x,y)^2}.
\]
A 2-plane $\mu$ is called \emph{holomorphic} (resp., \emph{totally real}) if the condition $\mu= J_\al \mu$
(resp., $\mu \perp J_\al \mu \neq \mu$ with respect to $g$) holds. The sectional curvature of a holomorphic (resp., totally real) 2-plane is called \emph{holomorphic} (resp., \emph{totally real}) \emph{sectional curvature}.
The 2-plane $\mu$ and its sectional curvature $k(\mu;p)$ are called a \emph{basic 2-plane} and a \emph{basic sectional curvature}, respectively, if $\mu$ has a basis $\{e_i,e_j\}$ $(i,j\in\{1,2,\dots,4n\},i\neq j)$ for a basis $\{e_1,e_2,\ldots,e_{4n}\}$ of $T_p\MM$. In the latter case we denote $k_{ij}$.

\section{Four-dimensional indecomposable real Lie algebras}\label{sect-lie}
Different authors study real 4-dimensional indecomposable Lie algebras. Firstly, a classification is given in \cite{Muba}, which could be be found easily in \cite{Pat} and \cite{GhaTho}.
The object of investigation in \cite{Andrada} are four-dimensional solvable real Lie algebras. The authors of this work establish the one-to-one correspondence between their classification and the classifications in \cite{Muba} and \cite{Pat}. In all of the cited works, the basic classes are described by the non-zero Lie brackets with respect to a basis $\{e_1,e_2,e_3,e_4\}$.
In Table~\ref{tab:Correspondence}, it is shown the correspondence between the mentioned classifications.

{\renewcommand{\arraystretch}{1.2}
\begin{table}
\caption{Correspondence between some classifications of Lie algebras}
	\label{tab:Correspondence}
\begin{tabular}{|c|c|c|c|c|}
  \hline
\cite{GhaTho} & \,\cite{Andrada} & \cite{Muba} & \cite{Pat} \\ \hhline{|=|=|=|=|=|}
$\mathfrak{g}_{4,1}$ & $\mathfrak{n}_{4}$ & $\mathfrak{g}_{4,1}$ & $A_{4,1}$ \\\hline
$\mathfrak{g}_{4,2}$ & $\mathfrak{r}_{4,a}$ & $\mathfrak{g}_{4,2}$ & $A_{4,2}^{a}$ \\\hline
$\mathfrak{g}_{4,3}$ & $\mathfrak{r}_{4,0}$ & $\mathfrak{g}_{4,3}$ & $A_{4,3}$ \\\hline
$\mathfrak{g}_{4,4}$ & $\mathfrak{r}_{4}$ & $\mathfrak{g}_{4,4}$ & $A_{4,4}$ \\\hline
$\mathfrak{g}_{4,5}$ & $\mathfrak{r}_{4,a,b}$ & $\mathfrak{g}_{4,5}$ & $A_{4,5}^{a,b}$ \\\hline
$\mathfrak{g}_{4,6}$ & $\mathfrak{r'}_{4,a,b}$ & $\mathfrak{g}_{4,6}$ & $A_{4,6}^{a,b}$ \\\hline
$\mathfrak{g}_{4,7}$ & $\mathfrak{h}_{4}$ & $\mathfrak{g}_{4,7}$ & $A_{4,7}$ \\\hline
$\mathfrak{g}_{4,8}$ & $\mathfrak{d}_{4}$ & $\mathfrak{g}_{4,8 (-1)}$  & $A_{4,8}$ \\\hline
$\mathfrak{g}_{4,9}$ & $\mathfrak{d}_{4,1/1+b}$ & $\mathfrak{g}_{4,8}$ & $A_{4,9}^{b}$ \\\hline
$\mathfrak{g}_{4,10}$ & $\mathfrak{d'}_{4,0}$ & $\mathfrak{g}_{4,9 (0)}$  & $A_{4,10}$ \\\hline
$\mathfrak{g}_{4,11}$ & $\mathfrak{d'}_{4,a}$ & $\mathfrak{g}_{4,9}$ & $A_{4,11}^{a}$ \\\hline
$\mathfrak{g}_{4,12}$ & $\mathfrak{aff}(\mathbb{C})$ & $\mathfrak{g}_{4,10}$ & $A_{4,12}$ \\\hline
\end{tabular}
\end{table}}

In the present work, we use the notation of the classes from \cite{GhaTho}, namely
\begin{subequations}\label{4i}
\begin{equation}
\begin{array}{llll}
\mathfrak{g}_{4,1}:\quad  & [e_2,e_4]=e_1, \quad         & [e_3,e_4]=e_2;         &\\[3pt]
\mathfrak{g}_{4,2}:\quad  & [e_1,e_4]=m e_1, \quad       & [e_2,e_4]=e_2,          &\\[3pt]
                          & [e_3,e_4]=e_2+e_3, \quad     &                        &(m\neq0); \\[3pt]
\mathfrak{g}_{4,3}:\quad  & [e_1,e_4]=e_1, \quad         & [e_3,e_4]=e_2;         &\\[3pt]
\mathfrak{g}_{4,4}:\quad  & [e_1,e_4]=e_1, \quad         & [e_2,e_4]=e_1+e_2,     &\\[3pt]
                          & [e_3,e_4]=e_2+e_3; \quad     &                        &\\[3pt]
\mathfrak{g}_{4,5}:\quad  & [e_1,e_4]=e_1, \quad         & [e_2,e_4]=a_1 e_2,     &\\[3pt]
                          & [e_3,e_4]=a_2 e_3, \quad     &                        &(a_1\neq0, a_2\neq0); \\[3pt]
\mathfrak{g}_{4,6}:\quad  & [e_1,e_4]=b_1 e_1, \quad     & [e_2,e_4]=b_2 e_2-e_3, &\\[3pt]
                          & [e_3,e_4]=e_2+b_2 e_3, \quad &                        &(b_1\neq0, b_2\geq0); \\[3pt]
\mathfrak{g}_{4,7}:\quad  & [e_1,e_4]=2 e_1, \quad       & [e_2,e_3]=e_1,         &\\[3pt]
                          & [e_2,e_4]=e_2, \quad         & [e_3,e_4]=e_2+e_3;     &\\[3pt]
\mathfrak{g}_{4,8}:\quad  & [e_2,e_3]=e_1, \quad         & [e_2,e_4]=e_2,         &\\[3pt]
                          & [e_3,e_4]=-e_3; \quad        &                        &\\[3pt]
\end{array}
\end{equation}
\begin{equation}
\begin{array}{llll}
\mathfrak{g}_{4,9}:\quad  & [e_1,e_4]=(p+1) e_1, \quad   & [e_2,e_3]=e_1,         &\\[3pt]
                          & [e_2,e_4]=e_2, \quad         & [e_3,e_4]=p e_3,       &(-1 < p \leq 1); \\[3pt]
\mathfrak{g}_{4,10}:\quad & [e_2,e_3]=e_1, \quad         & [e_2,e_4]=-e_3,        &\\[3pt]
                          & [e_3,e_4]=e_2; \quad         &                        &\\[3pt]
\mathfrak{g}_{4,11}:\quad & [e_1,e_4]=2q e_1, \quad      & [e_2,e_3]=e_1,         &\\[3pt]
                          & [e_2,e_4]=q e_2 - e_3, \quad & [e_3,e_4]=e_2 + q e_3, &(q > 0);\\[3pt]
\mathfrak{g}_{4,12}:\quad & [e_1,e_3]=e_1, \quad         & [e_1,e_4]=-e_2,        &\\[3pt]
                          & [e_2,e_3]=e_2, \quad         & [e_2,e_4]=e_1,         &
\end{array}
\end{equation}
\end{subequations}
where $a_1$, $a_2$, $b_1$, $b_2$, $m$, $p$, $q$ $\in \R$.

\section{Lie groups as almost hypercomplex manifolds with Hermitian-Norden metrics}\label{sect-corr}
Let $\mathcal{L}$ be a simply connected 4-dimensional real Lie group with corresponding Lie algebra $\mathfrak{l}$.
A standard hypercomplex structure on $\mathfrak{l}$ for its basis $\{e_1,e_2,e_3,e_4\}$ is defined as in \cite{So}:
\begin{equation}\label{JJJ}
\begin{array}{llll}
J_1e_1=e_2, \quad & J_1e_2=-e_1,  \quad &J_1e_3=-e_4, \quad &J_1e_4=e_3;
\\[6pt]
J_2e_1=e_3, &J_2e_2=e_4, &J_2e_3=-e_1, &J_2e_4=-e_2;
\\[6pt]
J_3e_1=-e_4, &J_3e_2=e_3, &J_3e_3=-e_2, &J_3e_4=e_1.
\end{array}
\end{equation}

Let $g$ be a pseudo-Riemannian metric of neutral signature for $x(x^1,x^2,\allowbreak{}x^3,x^4)$, $y(y^1,y^2,y^3,y^4) \in \mathfrak{l}$ defined by:
\begin{equation*}\label{g}
g(x,y)=x^1y^1+x^2y^2-x^3y^3-x^4y^4.
\end{equation*}
Bearing in mind the latter equality, it is valid that
\begin{equation}\label{gij}
\begin{array}{c}
g(e_1,e_1)=g(e_2,e_2)=-g(e_3,e_3)=-g(e_4,e_4)=1,\\[6pt]
g(e_i,e_j)=0,\;\; i\neq j \in \{1,2,3,4\}.
\end{array}
\end{equation}
Let us note that further the indices $i$, $j$, $k$, $l$ run over the range $\{1,2,3,4\}$.
The metric $g$ generates an almost hypercomplex structure with Hermitian-Norden metrics on $\mathfrak{l}$, according to \eqref{gJJ} and \eqref{gJ}. Then, $(\mathcal{L},H,G)$ is an almost hypercomplex manifold with Hermitian-Norden metrics.

\begin{thm}\label{thm-corr}
Let $(\mathcal{L},H,G)$ be a 4-dimensional almost hypercomplex manifold with Hermitian-Norden metrics. Then, the manifold $(\mathcal{L},H,G)$, which is corresponding to the different classes of 4-dimensional Lie algebras $\mathfrak{g}_{4,i}$, $(i=1,\dots,12)$, belongs to a certain class regarding $J_\al$ given in Table~\ref{tab:thm}, where we denote for brevity $\W_{i}\oplus\W_{j}$ and $\W_{i}\oplus\W_{j}\oplus\W_{k}$ by $\W_{ij}$ and $\W_{ijk}$, respectively.
{\renewcommand{\arraystretch}{1.2}
\begin{table}
\caption{Correspondence between different classes Lie algebras and the classes almost hypercomplex manifold with Hermitian-Norden metrics}
	\label{tab:thm}
\begin{tabular}{|c|c|c|c|c|}
  \hline
  Lie algebra & Parameters & $J_1$ & $J_2$ & $J_3$ \\ \hline
  $\mathfrak{g}_{4,1}$ & -- & $\W_{24}$ & $\W_{123}$ & $\W_{123}$ \\ \hline
  \multirow{2}{*}{$\mathfrak{g}_{4,2}$} & $m=1$ & $\W_4$ & $\W_{123}$ & $\W_{123}$ \\\cline{2-5}
   & $m\neq 0$; $m\neq 1$ & $\W_{24}$ & $\W_{123}$ & $\W_{123}$ \\  \hline
  $\mathfrak{g}_{4,3}$ & -- & $\W_{24}$ & $\W_{123}$ & $\W_{123}$ \\ \hline
  $\mathfrak{g}_{4,4}$ & -- & $\W_{24}$ & $\W_{123}$ & $\W_{123}$ \\ \hline
  \multirow{12}{*}{$\mathfrak{g}_{4,5}$}
            & $a_1=-1$, $a_2=1$ & $\W_2$ & $\W_{2}$ & $\W_{123}$ \\\cline{2-5}
            & $a_1=-1$, $a_2=-1$ & $\W_2$ & $\W_{123}$ & $\W_{2}$ \\\cline{2-5}
            & $a_1=-1$, $a_2\neq\pm1$ & $\W_2$ & $\W_{123}$ & $\W_{123}$ \\\cline{2-5}
            & $a_1=1$, $a_2=1$ & $\W_4$ & $\W_{1}$ & $\W_{12}$ \\\cline{2-5}
            & $a_1=1$, $a_2=-3$ & $\W_4$ & $\W_{23}$ & $\W_{23}$ \\\cline{2-5}
            & $a_1=-1$, $a_2\neq\{-3,1\}$ & $\W_4$ & $\W_{123}$ & $\W_{123}$ \\\cline{2-5}
            & $a_1\neq\pm1$, $a_2=1$ & $\W_{24}$ & $\W_{12}$ & $\W_{123}$ \\\cline{2-5}
            & $a_1=-\frac{1}{3}$, $a_2=-\frac{1}{3}$ & $\W_{24}$ & $\W_{23}$ & $\W_{12}$ \\\cline{2-5}
            & $a_1=-\frac{1}{2}(a_2+1)$, $a_2\neq\{-3,-\frac{1}{3},1\}$ & $\W_{24}$ & $\W_{23}$ & $\W_{123}$ \\\cline{2-5}
            & $a_1=a_2$, $a_2\neq\{\pm1,-\frac{1}{3}\}$ & $\W_{24}$ & $\W_{123}$ & $\W_{12}$ \\\cline{2-5}
            & $a_1=-a_2-2$, $a_2\neq\{-3,-1\}$ & $\W_{24}$ & $\W_{123}$ & $\W_{23}$ \\\cline{2-5}
            & $a_1\neq0$, $a_2\neq0$ & $\W_{24}$ & $\W_{123}$ & $\W_{123}$ \\\hline
  $\mathfrak{g}_{4,6}$ & $b_1\neq0$, $b_2\geq0$ & $\W_{24}$ & $\W_{123}$ & $\W_{12}$ \\ \hline
  $\mathfrak{g}_{4,7}$ & -- & $\W_{4}$ & $\W_{123}$ & $\W_{123}$ \\ \hline
  $\mathfrak{g}_{4,8}$ & -- & $\W_{24}$ & $\W_{123}$ & $\W_{3}$ \\ \hline
  \multirow{2}{*}{$\mathfrak{g}_{4,9}$} & $p=1$ & $\W_4$ & $\W_{12}$ & $\W_{12}$ \\\cline{2-5}
   & $-1<p<1$ & $\W_{24}$ & $\W_{12}$ & $\W_{123}$ \\  \hline
  $\mathfrak{g}_{4,10}$ & -- & $\W_{24}$ & $\W_{123}$ & $\W_{123}$ \\ \hline
  $\mathfrak{g}_{4,11}$ & $q>0$ & $\W_{24}$ & $\W_{123}$ & $\W_{12}$ \\
  \hline
  $\mathfrak{g}_{4,12}$ & -- & $\W_{4}$ & $\W_{123}$ & $\W_{123}$ \\ \hline
\end{tabular}
\end{table}}

Moreover, we have:
\begin{itemize}
  \item for each $a_1\neq0$ and $a_2\neq0$, $(\mathcal{L},H,G)$ does not belong to neither of $\W_0$ for $J_1$; $\W_0$, $\W_3$, $\W_1\oplus\W_3$ for $J_2$; $\W_0$, $\W_1$, $\W_3$, $\W_1\oplus\W_3$ for $J_3$;
  \item for each $b_1\neq0$, $b_2\geq0$, $(\mathcal{L},H,G)$ does not belong to neither of $\W_0$, $\W_2$, $\W_4$ for $J_1$; $\W_0$, $\W_1$, $\W_2$, $\W_3$, $\W_1\oplus\W_2$, $\W_1\oplus\W_3$, $\W_2\oplus\W_3$ for $J_2$; $\W_0$, $\W_1$, $\W_2$ for $J_3$;
  \item for each $m\neq0$, $(\mathcal{L},H,G)$ does not belong to neither of $\W_0$, $\W_2$ for $J_1$; $\W_0$, $\W_1$, $\W_2$, $\W_3$, $\W_1\oplus\W_2$, $\W_1\oplus\W_3$, $\W_2\oplus\W_3$ for $J_2$; $\W_0$, $\W_1$, $\W_2$, $\W_3$, $\W_1\oplus\W_2$, $\W_1\oplus\W_3$, $\W_2\oplus\W_3$ for $J_3$;
  \item for each $-1<p\leq1$, $(\mathcal{L},H,G)$ does not belong to neither of $\W_0$, $\W_2$ for $J_1$; $\W_0$, $\W_1$, $\W_2$ for $J_2$; $\W_0$, $\W_1$, $\W_2$, $\W_3$, $\W_1\oplus\W_3$, $\W_2\oplus\W_3$ for $J_3$;
  \item for each $q>0$, $(\mathcal{L},H,G)$ does not belong to neither of $\W_0$, $\W_2$, $\W_4$ for $J_1$; $\W_0$, $\W_1$, $\W_2$, $\W_3$, $\W_1\oplus\W_2$, $\W_1\oplus\W_3$, $\W_2\oplus\W_3$ for $J_2$; $\W_0$, $\W_1$, $\W_2$ for $J_3$.
\end{itemize}
\end{thm}

\begin{proof}
Now, we give our arguments for the case when the corresponding Lie algebra of $(\mathcal{L},H,G)$ is from $\mathfrak{g}_{4,1}$. Then, using \eqref{gJJ}, \eqref{4i}, \eqref{JJJ} and the well-known Koszul equality
\[
2g\left(\n_{e_i}e_j,e_k\right)
=g\left([e_i,e_j],e_k\right)+g\left([e_k,e_i],e_j\right)
+g\left([e_k,e_j],e_i\right),
\]
we obtain the components of the Levi-Civita connection $\n$ for the considered basis.
The non-zero of them are:
\begin{equation}\label{nabla41}
\begin{array}{l}
\n_{e_1}e_2=\n_{e_2}e_1=\n_{e_2}e_3=\n_{e_3}e_2=\frac{1}{2}e_4,\\[3pt]
\n_{e_1}e_4=\n_{e_3}e_4=\n_{e_4}e_1=-\n_{e_4}e_3=\frac{1}{2}e_2,\\[3pt]
\n_{e_2}e_4=\frac{1}{2}(e_1-e_3), \qquad \n_{e_4}e_2=-\frac{1}{2}(e_1+e_3).
\end{array}
\end{equation}
Then, we obtain the basic components $(F_{\al})_{ijk}=F_{\al}(e_i,e_j,e_k)$ of $F_{\al}$ by virtue of  \eqref{F'-al}, \eqref{JJJ}, \eqref{gij} and \eqref{nabla41}.
The non-zero of them are determined by the following ones and properties \eqref{FaJ-prop}
\begin{equation}\label{Fijk-41}
\begin{array}{l}
(F_{1})_{141}=(F_{1})_{213}=(F_{1})_{341}=(F_{1})_{413}=(F_{2})_{212}=(F_{2})_{223}=(F_{2})_{414} \\[3pt]
\phantom{(F_{1})_{141}}=-(F_{2})_{412}=\frac{1}{2}(F_{2})_{122}=\frac{1}{2}(F_{2})_{322}=(F_{3})_{134}=-(F_{3})_{213}\\[3pt]
\phantom{(F_{1})_{141}}=(F_{3})_{334}=(F_{3})_{413}=-\frac{1}{2}(F_{3})_{211}=-\frac{1}{2}(F_{3})_{422}=\frac{1}{2}.
\end{array}
\end{equation}
Using \eqref{theta-al} and \eqref{Fijk-41}, we establish the basic components $(\theta_{\al})_i=(\theta_{\al})(e_i)$ of the corresponding Lee forms and the non-zero are
\begin{equation*}\label{ta-i-41}
(\ta_1)_2=(\ta_2)_3=-(\ta_3)_2=-(\ta_3)_4=1.
\end{equation*}

After that, bearing in mind the classification conditions \eqref{cl-H-dim4} and \eqref{cl-N-dim4} for dimension 4, we conclude that in this case the manifold $(\mathcal{L},H,G)$ belongs to
\[
(\W_2\oplus\W_4)(J_1) \cap (\W_1\oplus\W_2\oplus\W_3)(J_2) \cap (\W_1\oplus\W_2\oplus\W_3)(J_3).
\]

The proofs for the cases of the classes $\mathfrak{g}_{4,2}$, $\mathfrak{g}_{4,5}$, $\mathfrak{g}_{4,6}$, $\mathfrak{g}_{4,9}$ and $\mathfrak{g}_{4,11}$ are given in \cite{HM12} and \cite{HM13}.

In a similar way we prove the assertions for the other classes using the following results for each case:
\begin{subequations}\label{res-F-4i}
\begin{equation*}
\begin{array}{l}
\mathfrak{g}_{4,3}:\\[3pt]
\n_{e_1}e_1=2\n_{e_2}e_3=2\n_{e_3}e_2=e_4, \quad \n_{e_1}e_4=e_1,\quad
\n_{e_2}e_4=\n_{e_4}e_2=-\frac{1}{2}e_3, \\[3pt]
\n_{e_3}e_4=-\n_{e_4}e_3=\frac{1}{2}e_2; \quad
(F_{1})_{113}=-2(F_{1})_{314}=2(F_{1})_{413}=1, \\[3pt]
(F_{2})_{112}=(F_{2})_{322}=2(F_{2})_{223}=-2(F_{2})_{412}=1, \\[3pt]
\frac{1}{2}(F_{3})_{111}=2(F_{3})_{213}=2(F_{3})_{312}=(F_{3})_{422}=-1; \\[3pt]
(\ta_1)_2=(\ta_1)_3=(\ta_2)_2=(\ta_2)_3=-\frac{1}{2}(\ta_3)_1=-(\ta_3)_4=1;\\[6pt]
\mathfrak{g}_{4,4}:\\[3pt]
\n_{e_1}e_1=2\n_{e_1}e_2=2\n_{e_2}e_1=\n_{e_2}e_2=2\n_{e_2}e_3=2\n_{e_3}e_2=-\n_{e_3}e_3=e_4,\\[3pt]
\n_{e_1}e_4=e_1+\frac{1}{2}e_2, \quad \n_{e_2}e_4=\frac{1}{2}e_1+e_2-\frac{1}{2}e_3,  \quad \n_{e_3}e_4=\frac{1}{2}e_2+e_3,\\[3pt]
\n_{e_4}e_1=-\n_{e_4}e_3=\frac{1}{2}e_2, \quad \n_{e_4}e_2=-\frac{1}{2}e_1-\frac{1}{2}e_3; \\[3pt]
\frac{1}{2}(F_{1})_{113}=-(F_{1})_{114}=(F_{1})_{213}=-\frac{1}{2}(F_{1})_{214}=-(F_{1})_{314}=(F_{1})_{413}=\frac{1}{2}, \\[3pt]
(F_{2})_{112}=(F_{2})_{122}=2(F_{2})_{212}=-2(F_{2})_{214}=\frac{1}{2}(F_{2})_{222}\\[3pt]
\phantom{\frac{1}{2}(F_{2})_{113}}                       =(F_{2})_{314}=(F_{2})_{322}=-2(F_{2})_{412}=2(F_{2})_{414}=1, \\[3pt]
\frac{1}{2}(F_{3})_{111}=2(F_{3})_{112}=(F_{3})_{211}=(F_{3})_{212}=2(F_{3})_{213}\\[3pt]
\phantom{\frac{1}{2}(F_{3})_{113}}                       =2(F_{3})_{312}=-(F_{3})_{313}=-2(F_{3})_{413}=(F_{3})_{422}=-1; \\[3pt]
2(\ta_1)_2=(\ta_1)_3=\frac{1}{2}(\ta_2)_2=2(\ta_2)_3=-\frac{1}{2}(\ta_3)_1=-2(\ta_3)_2=-2(\ta_3)_4=2;\\[6pt]
\mathfrak{g}_{4,7}:\\[3pt]
\frac{1}{2}\n_{e_1}e_1=\n_{e_2}e_2=-\n_{e_3}e_3=e_4, \quad \n_{e_1}e_2=\n_{e_2}e_1=-\n_{e_4}e_2=\frac{1}{2}e_3, \\[3pt]
\end{array}
\end{equation*}
\begin{equation*}
\begin{array}{llll}
\n_{e_1}e_3=\n_{e_3}e_1=-\n_{e_4}e_3=\frac{1}{2}e_2, \quad \n_{e_1}e_4=2e_1, \quad \n_{e_2}e_3=\frac{1}{2}e_1+\frac{1}{2}e_4, \\[3pt]
\n_{e_2}e_4=e_2-\frac{1}{2}e_3, \quad \n_{e_3}e_2=-\frac{1}{2}e_1+\frac{1}{2}e_4, \quad \n_{e_3}e_4=\frac{1}{2}e_2+e_3; \\[3pt]
(F_{1})_{113}=-(F_{1})_{214}=-3(F_{1})_{314}=3(F_{1})_{413}=\frac{3}{2},\\[3pt]
\frac{2}{5}(F_{2})_{112}=(F_{2})_{211}=-2(F_{2})_{214}=\frac{1}{2}(F_{2})_{222}\\[3pt]
\phantom{\frac{1}{2}(F_{2})_{113}}                       =\frac{2}{3}(F_{2})_{314}=(F_{2})_{322}=-2(F_{2})_{412}=1, \\[3pt]
\frac{1}{4}(F_{3})_{111}=-(F_{3})_{122}=2(F_{3})_{212}=2(F_{3})_{213}\\[3pt]
\phantom{\frac{1}{2}(F_{3})_{113}}                       =2(F_{3})_{312}=-\frac{2}{3}(F_{3})_{313}=(F_{3})_{422}=-1; \\[3pt]
(\ta_1)_2=\frac{1}{3}(\ta_1)_3=\frac{1}{6}(\ta_2)_2=(\ta_2)_3
                          =-\frac{1}{4}(\ta_3)_1=-(\ta_3)_4=1;\\[6pt]
\mathfrak{g}_{4,8}:\\[3pt]
\n_{e_1}e_2=\n_{e_2}e_1=-\frac{1}{2}\n_{e_3}e_4=\frac{1}{2}e_3, \quad
\n_{e_1}e_3=\frac{1}{2}\n_{e_2}e_4=\n_{e_3}e_1=\frac{1}{2}e_2, \\[3pt]
\n_{e_2}e_2=\n_{e_3}e_3=e_4, \quad \n_{e_2}e_3=-\n_{e_3}e_2=\frac{1}{2}e_1; \\[3pt]
(F_{1})_{113}=\frac{1}{3}(F_{1})_{214}=-\frac{1}{2}, \quad
(F_{3})_{122}=-2(F_{3})_{212}=-2(F_{3})_{313}=1, \\[3pt]
2(F_{2})_{112}=(F_{2})_{211}=\frac{1}{2}(F_{2})_{222}=-2(F_{2})_{314}=1; \quad
(\ta_1)_3=\frac{1}{2}(\ta_2)_2=1;\\[6pt]
\mathfrak{g}_{4,10}:\\[3pt]
\n_{e_1}e_2=\n_{e_2}e_1=-\frac{1}{2}\n_{e_2}e_4=\frac{1}{2}e_3, \quad
\n_{e_1}e_3=\n_{e_3}e_1=\frac{1}{2}\n_{e_3}e_4=\frac{1}{2}e_2, \\[3pt]
\n_{e_2}e_3=\frac{1}{2}e_1+e_4, \quad \n_{e_3}e_2=-\frac{1}{2}e_1+e_4; \\[3pt]
(F_{1})_{113}=(F_{1})_{214}=\frac{1}{2}(F_{1})_{314}=-\frac{1}{2},\\[3pt]
(F_{2})_{112}=\frac{1}{2}(F_{2})_{211}=-\frac{1}{2}(F_{2})_{214}=(F_{2})_{314}=\frac{1}{4}(F_{2})_{322}=\frac{1}{2}, \\[3pt]
(F_{3})_{122}=2(F_{3})_{212}=-(F_{3})_{213}=-(F_{3})_{312}=2(F_{3})_{313}=1; \\[3pt]
(\ta_1)_2=(\ta_2)_2=(\ta_2)_3=-\frac{1}{2}(\ta_3)_4=1;\\[6pt]
\mathfrak{g}_{4,12}:\\[3pt]
\n_{e_1}e_1=\n_{e_2}e_2=e_3, \quad \n_{e_1}e_3=-\n_{e_4}e_2=e_1, \quad \n_{e_2}e_3=\n_{e_4}e_1=e_2, \\[3pt]
(F_{1})_{114}=\frac{3}{2}(F_{1})_{213}=-\frac{3}{2}, \quad \frac{1}{2}(F_{2})_{111}=(F_{2})_{212}=(F_{2})_{414}=1, \\[3pt]
(F_{3})_{112}=\frac{1}{2}(F_{3})_{222}=(F_{3})_{413}=1, \quad
(\ta_1)_4=-(\ta_2)_1=-(\ta_3)_2=-2.
\end{array}
\end{equation*}
\end{subequations}

\end{proof}

\section{Curvature properties of the manifolds under study}\label{sect-prop}
In this section we determine some geometric characteristics of the manifolds $(\mathcal{L},H,G)$ in all the classes considered in the previous section.
The focus of the con\-si\-de\-rations in \cite{HM12} and \cite{HM13} are the classes of the classification of 4-dimensional indecomposable real Lie algebras, given in \eqref{4i}, depending on real parameters.
Actually, these five classes are families of manifolds whose properties are functions of the parameters.
The curvature properties of the considered manifolds are summarized in the following

\begin{theorem}\label{thm-char}
Let $(\mathcal{L},H,G)$ be a 4-dimensional almost hypercomplex manifold with Hermitian-Norden metrics, as well as let the corresponding Lie algebra $\mathfrak{l}$ of $L$ be from the class $\mathfrak{g}_{4,i}$, $(i=1,\dots,12)$, given in \eqref{4i}. Then the
following propositions are valid:
\begin{enumerate}
    \item Every $(\mathcal{L},H,G)$ is non-flat;
    \item An $(\mathcal{L},H,G)$ is scalar flat if and only if $\mathfrak{l}$ belongs to:
            \begin{enumerate}
            \item $\mathfrak{g}_{4,1}$,
            \item $\mathfrak{g}_{4,6} \; \bigl(b_1=-b_2\pm\sqrt{1-2b_2^2}, \; 0\leq b_2 \leq \frac{\sqrt{2}}{2}, \; b_2 \neq \frac{\sqrt{3}}{3}\bigr)$,
            \item $\mathfrak{g}_{4,11} \; \bigl(q=\frac{\sqrt{3}}{6}\bigr)$;
            \end{enumerate}
    \item An $(\mathcal{L},H,G)$ has positive scalar curvature if and only if $\mathfrak{l}$ belongs to:
            \begin{enumerate}
            \item $\mathfrak{g}_{4,i} \; \bigl(i=2,3,4,5,7,8,9,11,12\bigr)$,
            \item $\mathfrak{g}_{4,6} \; \bigl(b_1\neq0, \; b_2 > \frac{\sqrt{2}}{2}\bigr)$,
            \item $\mathfrak{g}_{4,11} \; \bigl(q>\frac{\sqrt{3}}{6}\bigr)$;
            \end{enumerate}
    \item An $(\mathcal{L},H,G)$ has negative scalar curvature if and only if $\mathfrak{l}$ belongs to:
            \begin{enumerate}
            \item $\mathfrak{g}_{4,10}$,
            \item $\mathfrak{g}_{4,11} \; \bigl(0<q<\frac{\sqrt{3}}{6}\bigr)$;
            \end{enumerate}
    \item Every $(\mathcal{L},H,G)$ is $*$-scalar flat w.r.t. $J_1$ and $J_2$;
    \item An $(\mathcal{L},H,G)$ is $*$-scalar flat w.r.t. $J_3$ if and only if $\mathfrak{l}$ belongs to:
            \begin{enumerate}
            \item $\mathfrak{g}_{4,i} \; \bigl(i=1,5,8,9,10,12\bigr)$,
            \item $\mathfrak{g}_{4,2} \; \bigl(m=-2\bigr)$,
            \item $\mathfrak{g}_{4,6} \; \bigl(b_1=-2b_2\bigr)$;
            \end{enumerate}
    \item An $(\mathcal{L},H,G)$ is $**$-scalar flat w.r.t. $J_1$ if and only if $\mathfrak{l}$ belongs to:
            \begin{enumerate}
            \item $\mathfrak{g}_{4,2} \; \bigl(m=-\frac{1}{4}\bigr)$,
            \item $\mathfrak{g}_{4,5} \; \bigl(a_1=-a_2^2\bigr)$,
            \item $\mathfrak{g}_{4,6} \; \bigl(b_1=b_2^{-1}-b_2\bigr)$,
            \item $\mathfrak{g}_{4,11} \; \bigl(q=\frac{\sqrt{15}}{6}\bigr)$;
            \end{enumerate}
    \item An $(\mathcal{L},H,G)$ is $**$-scalar flat w.r.t. $J_2$ if and only if $\mathfrak{l}$ belongs to:
            \begin{enumerate}
            \item $\mathfrak{g}_{4,2} \; \bigl(m=-\frac{5}{4}\bigr)$,
            \item $\mathfrak{g}_{4,5} \; \bigl(a_2=-a_1^2\bigr)$,
            \item $\mathfrak{g}_{4,6} \; \bigl(b_1=b_2^{-1}-b_2\bigr)$,
            \item $\mathfrak{g}_{4,11} \; \bigl(q=\frac{\sqrt{15}}{6}\bigr)$;
            \end{enumerate}
    \item An $(\mathcal{L},H,G)$ is $**$-scalar flat w.r.t. $J_3$ if and only if $\mathfrak{l}$ belongs to:
            \begin{enumerate}
            \item $\mathfrak{g}_{4,1}$,
            \item $\mathfrak{g}_{4,9} \; \bigl(p=\frac{\sqrt{2}-3}{2}\bigr)$;
            \end{enumerate}
    \item An $(\mathcal{L},H,G)$ has positive basic holomorphic sectional curvatures w.r.t. $J_1$ (\ie $k_{12}$ and $k_{34}$) if and only if $\mathfrak{l}$ belongs to:
            \begin{enumerate}
            \item $\mathfrak{g}_{4,i} \; \bigl(i=4,7\bigr)$,
            \item $\mathfrak{g}_{4,2} \; \bigl(m>0\bigr)$,
            \item $\mathfrak{g}_{4,5} \; \bigl(a_1>0\bigr)$,
            \item $\mathfrak{g}_{4,6} \; \bigl(b_1>0,\; b_2>1\bigr)$,
            \item $\mathfrak{g}_{4,9} \; \bigl(-\frac{3}{4}<p\leq1,\; p\neq0\bigr)$,
            \item $\mathfrak{g}_{4,11} \; \bigl(q>1\bigr)$;
            \end{enumerate}
    \item An $(\mathcal{L},H,G)$ has positive basic holomorphic sectional curvatures w.r.t. $J_2$ (\ie $k_{13}$ and $k_{24}$) if and only if $\mathfrak{l}$ belongs to:
            \begin{enumerate}
            \item $\mathfrak{g}_{4,i} \; \bigl(i=4,7\bigr)$,
            \item $\mathfrak{g}_{4,2} \; \bigl(m>0\bigr)$,
            \item $\mathfrak{g}_{4,5} \; \bigl(a_2>0\bigr)$,
            \item $\mathfrak{g}_{4,6} \; \bigl(b_1>0,\; b_2>1\bigr)$,
            \item $\mathfrak{g}_{4,9} \; \bigl(\frac{\sqrt{2}-1}{2}<p\leq1\bigr)$,
            \item $\mathfrak{g}_{4,11} \; \bigl(q>1\bigr)$;
            \end{enumerate}
    \item An $(\mathcal{L},H,G)$ has positive basic holomorphic sectional curvatures w.r.t. $J_3$ (\ie $k_{14}$ and $k_{23}$) if and only if $\mathfrak{l}$ belongs to:
            \begin{enumerate}
            \item $\mathfrak{g}_{4,i} \; \bigl(i=2,3,4,6,7,11\bigr)$,
            \item $\mathfrak{g}_{4,5} \; \bigl(a_1a_2>0\bigr)$,
            \item $\mathfrak{g}_{4,9} \; \bigl(-\frac{3}{4}<p\leq1,\; p\neq0\bigr)$;
            \end{enumerate}
    \item An $(\mathcal{L},H,G)$ has negative basic holomorphic sectional curvatures w.r.t. $J_1$ (\ie $k_{12}$ and $k_{34}$) if and only if $\mathfrak{l}$ belongs to:
            \begin{enumerate}
            \item $\mathfrak{g}_{4,1}$,
            \item $\mathfrak{g}_{4,6} \; \bigl(b_1<0,\; 0<b_2<1\bigr)$,
            \item $\mathfrak{g}_{4,11} \; \bigl(0<q<\frac{\sqrt{2}}{4}\bigr)$;
            \end{enumerate}
    \item An $(\mathcal{L},H,G)$ has negative basic holomorphic sectional curvatures w.r.t. $J_2$ (\ie $k_{13}$ and $k_{24}$) if and only if $\mathfrak{l}$ belongs to:
            \begin{enumerate}
            \item $\mathfrak{g}_{4,6} \; \bigl(b_1<0,\; 0<b_2<1\bigr)$,
            \item $\mathfrak{g}_{4,11} \; \bigl(0<q<\frac{\sqrt{2}}{4}\bigr)$;
            \end{enumerate}
    \item Every $(\mathcal{L},H,G)$ has non-negative basic holomorphic sectional curvatures w.r.t. $J_3$ (\ie $k_{14}$ and $k_{23}$);
    \item An $(\mathcal{L},H,G)$ has positive basic totally real sectional curvatures w.r.t. $J_1$ (\ie $k_{13}$, $k_{14}$, $k_{23}$ and $k_{24}$) if and only if $\mathfrak{l}$ belongs to:
            \begin{enumerate}
            \item $\mathfrak{g}_{4,i} \; \bigl(i=4,7\bigr)$,
            \item $\mathfrak{g}_{4,2} \; \bigl(m>0\bigr)$,
            \item $\mathfrak{g}_{4,5} \; \bigl(a_1>0,\; a_2>0\bigr)$,
            \item $\mathfrak{g}_{4,6} \; \bigl(b_1>0,\; b_2>1\bigr)$,
            \item $\mathfrak{g}_{4,9} \; \bigl(\frac{\sqrt{2}-1}{2}<p\leq1\bigr)$,
            \item $\mathfrak{g}_{4,11} \; \bigl(q>1\bigr)$;
            \end{enumerate}
    \item An $(\mathcal{L},H,G)$ has positive basic totally real sectional curvatures w.r.t. $J_2$ (\ie $k_{12}$, $k_{14}$, $k_{23}$ and $k_{34}$) if and only if $\mathfrak{l}$ belongs to:
            \begin{enumerate}
            \item $\mathfrak{g}_{4,i} \; \bigl(i=4,7\bigr)$,
            \item $\mathfrak{g}_{4,2} \; \bigl(m>0\bigr)$,
            \item $\mathfrak{g}_{4,5} \; \bigl(a_1>0,\; a_2>0\bigr)$,
            \item $\mathfrak{g}_{4,6} \; \bigl(b_1>0,\; b_2>1\bigr)$,
            \item $\mathfrak{g}_{4,9} \; \bigl(-\frac{3}{4}<p\leq1,\; p\neq0\bigr)$,
            \item $\mathfrak{g}_{4,11} \; \bigl(q>1\bigr)$;
            \end{enumerate}
    \item An $(\mathcal{L},H,G)$ has positive basic totally real sectional curvatures w.r.t. $J_3$ (\ie $k_{12}$, $k_{13}$, $k_{24}$ and $k_{34}$) if and only if $\mathfrak{l}$ belongs to:
            \begin{enumerate}
            \item $\mathfrak{g}_{4,i} \; \bigl(i=4,7\bigr)$,
            \item $\mathfrak{g}_{4,2} \; \bigl(m>0\bigr)$,
            \item $\mathfrak{g}_{4,5} \; \bigl(a_1>0,\; a_2>0\bigr)$,
            \item $\mathfrak{g}_{4,6} \; \bigl(b_1>0,\; b_2>1\bigr)$,
            \item $\mathfrak{g}_{4,9} \; \bigl(\frac{\sqrt{2}-1}{2}<p\leq1\bigr)$,
            \item $\mathfrak{g}_{4,11} \; \bigl(q>1\bigr)$;
            \end{enumerate}
    \item Every $(\mathcal{L},H,G)$ has non-negative basic totally real sectional curvatures w.r.t. $J_1$ (\ie $k_{13}$, $k_{14}$, $k_{23}$ and $k_{24}$) and $J_2$ (\ie $k_{12}$, $k_{14}$, $k_{23}$ and $k_{34}$);
    \item An $(\mathcal{L},H,G)$ has negative basic totally real sectional curvatures w.r.t. $J_3$ (\ie $k_{12}$, $k_{13}$, $k_{24}$ and $k_{34}$) if and only if $\mathfrak{l}$ belongs to:
            \begin{enumerate}
            \item $\mathfrak{g}_{4,6} \; \bigl(b_1<0,\; 0<b_2<1\bigr)$,
            \item $\mathfrak{g}_{4,11} \; \bigl(0<q<\frac{\sqrt{2}}{4}\bigr)$.
            \end{enumerate}
    \end{enumerate}
\end{theorem}

\begin{proof}
Firstly, we present our proof for the case when the corresponding Lie algebra of $\mathcal{L}$ belongs to $\mathfrak{g}_{4,1}$.

Using \eqref{gJJ}, \eqref{R}, \eqref{JJJ} and the definition of $\mathfrak{g}_{4,1}$ in \eqref{4i}, we calculate the basic components $R_{ijkl}=R(e_i,e_j,e_k,e_l)$ of
$R$. The non-zero of them are determined by the following ones and properties \eqref{R-prop}
\begin{equation}\label{R-41}
\begin{array}{c}
R_{1212}=-R_{1223}=-R_{1414}=R_{1434}=R_{2323}=\frac{1}{4}R_{2424}=\frac{1}{3}R_{3434}=\frac{1}{4}.
\end{array}
\end{equation}

Bearing in mind the latter equalities, \eqref{gJJ}, \eqref{JJJ} and \eqref{gij}, we obtain the basic components $\rho_{jk}=\rho(e_j,e_k)$,  $(\rho_\al^*)_{jk}=\rho_\al^*(e_j,e_k)$, as well as the values of $\tau$, $\tau_\al^*$, $\tau_\al^{**}$ and $k_{ij}=k(e_i,e_j)$. Having in mind properties \eqref{rho-prop}, the non-zero of them are determined by
\begin{equation}\label{rho-tau-41}
\begin{array}{c}
\rho_{11}=-\frac{1}{2}\rho_{22}=-\rho_{33}=-\frac{1}{2}, \;\;
(\rho_1^*)_{12}=(\rho_1^*)_{14}=-(\rho_1^*)_{23}=\frac{1}{3}(\rho_1^*)_{34}=-\frac{1}{4}, \\[3pt]
(\rho_2^*)_{22}=-\frac{1}{2}(\rho_2^*)_{24}=(\rho_2^*)_{44}=-\frac{1}{2},\;\;
(\rho_3^*)_{12}=(\rho_3^*)_{14}=(\rho_3^*)_{23}=\frac{1}{4}, \\[3pt]
\tau_1^{**}=-\tau_2^{**}=-2, \;\; k_{12}=k_{14}=-k_{23}=-\frac{1}{4}k_{24}=\frac{1}{3}k_{34}=-\frac{1}{4}.
\end{array}
\end{equation}

By virtue of \eqref{R-41} and \eqref{rho-tau-41}, we establish the truthfullness of the statements for the case of $\mathfrak{g}_{4,1}$.

The results for the cases of the classes $\mathfrak{g}_{4,2}$, $\mathfrak{g}_{4,5}$, $\mathfrak{g}_{4,6}$, $\mathfrak{g}_{4,9}$ and $\mathfrak{g}_{4,11}$, which are are summarized here, are given in \cite{HM12} and \cite{HM13}.

In a similar way as for $\mathfrak{g}_{4,1}$, we obtain the following results for $(\mathcal{L},H,G)$ in the other cases and we prove the respective assertions:

\begin{subequations}
\begin{equation*}\label{res-R-4i}
\begin{array}{l}
\mathfrak{g}_{4,3}:\\[3pt]
-\frac{1}{2}R_{1213}=\frac{1}{4}R_{1414}=R_{2323}=R_{2424}=\frac{1}{3}R_{3434}=\frac{1}{4};\\[3pt]
\frac{1}{2}\rho_{11}=\rho_{22}=\rho_{23}=\rho_{33}=-\rho_{44}=\frac{1}{2}, \\[3pt]
-\frac{1}{3}(\rho_1^*)_{34}=-\frac{1}{2}(\rho_2^*)_{12}=(\rho_2^*)_{24}=\frac{1}{4}(\rho_3^*)_{11}
=-\frac{1}{4}(\rho_3^*)_{14}=(\rho_3^*)_{23}=\frac{1}{4}, \\[3pt]
\tau=\frac{3}{2}\tau_3^*=-\tau_1^{**}=3\tau_2^{**}=5\tau_3^{**}=\frac{3}{2}, \;\; \frac{1}{4}k_{14}=k_{23}=k_{24}=-\frac{1}{3}k_{34}=\frac{1}{4};\\[6pt]
\mathfrak{g}_{4,4}:\\[3pt]
-\frac{4}{3}R_{1212}=-2R_{1213}=-4R_{1223}=R_{1313}=2R_{1323}=\frac{4}{3}R_{1414}\\[3pt]
\phantom{-\frac{4}{3}R_{1212}}=R_{1424}=4R_{1434}=\frac{4}{5}R_{2323}=\frac{1}{2}R_{2424}=R_{2434}=-4R_{3434}=1;\\[3pt]
\frac{3}{5}\rho_{11}=\rho_{12}=\frac{3}{8}\rho_{22}=\rho_{23}=-\frac{3}{5}\rho_{33}=-\frac{1}{2}\rho_{44}=\frac{3}{2}, \\[3pt]
\frac{1}{3}(\rho_1^*)_{12}=\frac{1}{2}(\rho_1^*)_{13}=-(\rho_1^*)_{14}=(\rho_1^*)_{23}
=-\frac{1}{4}(\rho_1^*)_{24}=(\rho_1^*)_{34}=\frac{1}{4}, \\[3pt]
(\rho_2^*)_{12}=-\frac{1}{2}(\rho_2^*)_{13}=-\frac{1}{2}(\rho_2^*)_{14}=(\rho_2^*)_{22}=-(\rho_2^*)_{23}\\[3pt]
\phantom{(\rho_2^*)_{12}}=-\frac{1}{4}(\rho_2^*)_{24}=-\frac{1}{2}(\rho_2^*)_{34}=(\rho_2^*)_{44}=-\frac{1}{2}, \\[3pt]
(\rho_3^*)_{11}=4(\rho_3^*)_{12}=2(\rho_3^*)_{13}=-\frac{4}{3}(\rho_3^*)_{14}=\frac{4}{5}(\rho_3^*)_{23}\\[3pt]
\phantom{(\rho_3^*)_{11}}=-(\rho_3^*)_{24}=-4(\rho_3^*)_{34}=-\frac{1}{2}(\rho_3^*)_{44}=1, \\[3pt]
\tau=4\tau_3^*=6\tau_1^{**}=2\tau_2^{**}=3\tau_3^{**}=12, \\[3pt]
k_{12}=\frac{4}{3}k_{13}=k_{14}=\frac{3}{5}k_{23}=\frac{3}{8}k_{24}=3k_{34}=\frac{3}{4};\\[6pt]
\mathfrak{g}_{4,7}:\\[3pt]
-\frac{4}{7}R_{1212}=-R_{1213}=4R_{1224}=-2R_{1234}=\frac{4}{7}R_{1313}=2R_{1324}=4R_{1334}\\[3pt]
\phantom{-\frac{4}{3}R_{1212}}=\frac{1}{4}R_{1414}=R_{1423}=\frac{1}{2}R_{2323}=\frac{4}{5}R_{2424}=R_{2434}=-4R_{3434}=1;\\[3pt]
\frac{2}{15}\rho_{11}=\frac{1}{5}\rho_{22}=\frac{1}{2}\rho_{23}=-\frac{1}{4}\rho_{33}=-\frac{2}{11}\rho_{44}=1, \\[3pt]
\frac{1}{3}(\rho_1^*)_{12}=(\rho_1^*)_{13}=-\frac{3}{5}(\rho_1^*)_{24}=(\rho_1^*)_{34}=\frac{3}{4}, \\[3pt]
-\frac{4}{5}(\rho_2^*)_{12}=\frac{4}{13}(\rho_2^*)_{13}=\frac{4}{11}(\rho_2^*)_{24}=\frac{4}{3}(\rho_2^*)_{34}=1, \\[3pt]
\end{array}
\end{equation*}
\begin{equation*}
\begin{array}{llll}
(\rho_3^*)_{11}=-\frac{1}{2}(\rho_3^*)_{14}=-4(\rho_3^*)_{22}=(\rho_3^*)_{23}=-4(\rho_3^*)_{33}=-(\rho_3^*)_{44}=2, \\[3pt]
\frac{6}{11}\tau=3\tau_3^*=3\tau_1^{**}=2\tau_2^{**}=\tau_3^{**}=12, \\[3pt] k_{12}=k_{13}=\frac{7}{16}k_{14}=\frac{7}{8}k_{23}=\frac{7}{5}k_{24}=7k_{34}=\frac{7}{4};\\[6pt]
\mathfrak{g}_{4,8}:\\[3pt]
4R_{1212}=2R_{1234}=-4R_{1313}=2R_{1324}=-4R_{2323}=R_{2424}=-R_{3434}=1;\\[3pt]
\rho_{11}=-\rho_{22}=\rho_{33}=\frac{1}{4}\rho_{44}=-\frac{1}{2}, \\[3pt]
(\rho_1^*)_{12}=-\frac{3}{2}(\rho_1^*)_{34}=(\rho_2^*)_{13}=-\frac{3}{2}(\rho_2^*)_{24}=-\frac{3}{4}(\rho_3^*)_{14}
=\frac{3}{5}(\rho_3^*)_{23}=-\frac{3}{4}, \\[3pt]
\tau=\frac{5}{3}\tau_1^{**}=\frac{5}{3}\tau_2^{**}=-5\tau_3^{**}=\frac{5}{2}, \;\; k_{12}=k_{13}=k_{23}=-\frac{1}{4}k_{24}=-\frac{1}{4}k_{34}=-\frac{1}{4};\\[6pt]
\mathfrak{g}_{4,10}:\\[3pt]
4R_{1212}=2R_{1224}=-4R_{1313}=2R_{1334}=\frac{4}{7}R_{2323}=-R_{2424}=R_{3434}=1;\\[3pt]
\rho_{11}=-\rho_{22}=\rho_{33}=-\frac{1}{4}\rho_{44}=-\frac{1}{2}, \;\;
4(\rho_1^*)_{12}=2(\rho_1^*)_{13}=2(\rho_1^*)_{24}=(\rho_1^*)_{34}=-1, \\[3pt]
2(\rho_2^*)_{12}=4(\rho_2^*)_{13}=(\rho_2^*)_{24}=2(\rho_2^*)_{34}=-1, \;\;
(\rho_3^*)_{22}=-\frac{4}{7}(\rho_3^*)_{23}=(\rho_3^*)_{33}=-1, \\[3pt]
\tau=\frac{7}{10}\tau_1^{**}=\frac{7}{10}\tau_2^{**}=-\frac{1}{2}\tau_3^{**}=-\frac{7}{4}, \;\; k_{12}=k_{13}=-\frac{1}{7}k_{23}=\frac{1}{4}k_{24}=\frac{1}{4}k_{34}=-\frac{1}{4};\\[6pt]
\mathfrak{g}_{4,12}:\\[3pt]
R_{1212}=-R_{1313}=-R_{2323}=-1;\;\; \rho_{11}=\rho_{22}=-\rho_{33}=2, \\[3pt]
(\rho_1^*)_{12}=(\rho_2^*)_{13}=(\rho_3^*)_{23}=1, \;\;
\frac{1}{3}\tau=\tau_1^{**}=\tau_2^{**}=\tau_3^{**}=2, \;\; k_{12}=k_{13}=k_{23}=1.
\end{array}
\end{equation*}
\end{subequations}

\end{proof}


\end{document}